\documentclass[a4paper]{amsart}
\usepackage{verbatim}
\usepackage[dvips]{color}
\usepackage{amssymb}
\usepackage[active]{srcltx}
\usepackage{latexsym}
\usepackage{amsmath}
 \usepackage{float}
\usepackage{mathrsfs} %fonts veramente calligrafici con \mathscr{B}
\usepackage[draft]{fixme}

\allowdisplaybreaks

\newtheorem{theorem}{Theorem}[section]

\newtheorem{lemma}{Lemma}[section]
\newtheorem{corollary}{Corollary}[section]
\theoremstyle{definition}

\newtheorem{remark}{Remark}[section]

\floatstyle{ruled} \restylefloat{figure} \restylefloat{table}

\numberwithin{equation}{section}

\newcommand{\beq}{\begin{equation}}
\newcommand{\bea}[1]{\begin{array}{#1} }
\newcommand{\eeq}{ \end{equation}}
\newcommand{\ea}{ \end{array}}

\newcommand{\al}{\alpha}

\newcommand{\rar}{\mbox{$\rightarrow$}}

\newcommand{\he}{\theta}

\def \R  {{\mathbb {R}}}

\def \n {{\nu}}

\def\mean#1{\mathchoice%
          {\mathop{\kern 0.2em\vrule width 0.6em height 0.69678ex depth -0.58065ex
                  \kern -0.8em \intop}\nolimits_{\kern -0.4em#1}}%
          {\mathop{\kern 0.1em\vrule width 0.5em height 0.69678ex depth -0.60387ex
                  \kern -0.6em \intop}\nolimits_{#1}}%
          {\mathop{\kern 0.1em\vrule width 0.5em height 0.69678ex
              depth -0.60387ex
                  \kern -0.6em \intop}\nolimits_{#1}}%
          {\mathop{\kern 0.1em\vrule width 0.5em height 0.69678ex depth -0.60387ex
                  \kern -0.6em \intop}\nolimits_{#1}}}

\def\vintslides_#1{\mathchoice%
          {\mathop{\kern 0.1em\vrule width 0.5em height 0.697ex depth -0.581ex
                  \kern -0.6em \intop}\nolimits_{\kern -0.4em#1}}%
          {\mathop{\kern 0.1em\vrule width 0.3em height 0.697ex depth -0.604ex
                  \kern -0.4em \intop}\nolimits_{#1}}%
          {\mathop{\kern 0.1em\vrule width 0.3em height 0.697ex depth -0.604ex
                  \kern -0.4em \intop}\nolimits_{#1}}%
          {\mathop{\kern 0.1em\vrule width 0.3em height 0.697ex depth -0.604ex
                  \kern -0.4em \intop}\nolimits_{#1}}}

\newcommand{\aveint}[2]{\mathchoice%
          {\mathop{\kern 0.2em\vrule width 0.6em height 0.69678ex depth -0.58065ex
                  \kern -0.8em \intop}\nolimits_{\kern -0.45em#1}^{#2}}%
          {\mathop{\kern 0.1em\vrule width 0.5em height 0.69678ex depth -0.60387ex
                  \kern -0.6em \intop}\nolimits_{#1}^{#2}}%
          {\mathop{\kern 0.1em\vrule width 0.5em height 0.69678ex depth -0.60387ex
                  \kern -0.6em \intop}\nolimits_{#1}^{#2}}%
          {\mathop{\kern 0.1em\vrule width 0.5em height 0.69678ex depth -0.60387ex
                  \kern -0.6em \intop}\nolimits_{#1}^{#2}}}

\def\eqn#1$$#2$${\begin{equation}\label#1#2\end{equation}}
\def\charfn_#1{{\raise1.2pt\hbox{$\chi
_{\kern-1pt\lower3pt\hbox{{$\scriptstyle#1$}}}$}}}

\def\qq1{q_*}
\def\q2{q_{**}}

\def\er{\mathbb R}

\newdimen\vintbar
\def\vint{-\kern-\vintbar\int}

\def\0{\boldsymbol 0}

\newcommand{\divo}{\textnormal{div}}

\newtoks\by
\newtoks\paper
\newtoks\book
\newtoks\jour
\newtoks\yr
\newtoks\pages
\newtoks\vol
\newtoks\publ

\def\name[#1, #2]{#1 #2}
\def\ota{{\hbox{\bf ???}}}
\def\cLear{\by=\ota\paper=\ota\book=\ota\jour=\ota\yr=\ota
\pages=\ota\vol=\ota\publ=\ota}
\def\endpaper{\the\by, \textit{\the\paper},
{\the\jour} \textbf{\the\vol} (\the\yr), \the\pages.\cLear}
\def\endbook{\the\by, \textit{\the\book},
\the\publ, \the\yr.\cLear}
\def\endpap{\the\by, \textit{\the\paper}, \the\jour.\cLear}
\def\endproc{\the\by, \textit{\the\paper}, \the\book, \the\publ,
\the\yr, \the\pages.\cLear}

\begin{document}
\title[Cauchy integrals for the $p$-Laplace equation]{Cauchy integrals for the $p$-Laplace equation\\ in planar Lipschitz domains}

%\author[Nystr\"{o}m]{Kaj Nystr\"{o}m}
\address{Kaj Nystr\"{o}m\\Department of Mathematics, Uppsala University\\
S-751 06 Uppsala, Sweden}
\email{kaj.nystrom@math.uu.se}
%\author[Rosen]{Andreas Rosen}
\address{Andreas Ros\'en\\Mathematical Sciences, Chalmers University of Technology and University of Gothenburg\\
SE-412 96 G{\"o}teborg, Sweden}
\email{andreas.rosen@chalmers.se}

\author{K. Nystr{\"o}m, A. Ros\'en}

\maketitle

   \begin{abstract}
 \noindent We construct solutions to $p$-Laplace type equations in unbounded Lipschitz domains in the plane with prescribed boundary data in appropriate fractional Sobolev spaces. Our approach builds on a Cauchy integral representation formula for solutions.
 \\

    \noindent
   2000  {\em Mathematics Subject Classification.}  \\

    \end{abstract}

\setcounter{equation}{0} \setcounter{theorem}{0}
\section{Introduction}
\noindent In \cite{AAH}, \cite{AAMc}, \cite{AA}, \cite{AR}, new representations and new methods for solving boundary value problems for divergence form second order, real and complex, equations and systems were developed in domains Lipschitz diffeomorphic to the upper half space
$\mathbb R_+^{n+1}:=\{(x,t)\in\mathbb R^n\times \mathbb R:\ t>0\}$, $n\geq 1$. Focusing on the case of equations, the authors
consider equations
\begin{eqnarray}\label{g1}
Lu(x,t)=\sum_{i,j=1}^{n+1}\partial_i(a_{i,j}(x,t)\partial_ju(x,t))=0,\ \partial_{n+1}=\partial_t,\ \partial_i=\partial_{x_i},
\end{eqnarray}
with $A=A(x,t)=\{a_{i,j}(x,t)\}_{i,j=1}^{n+1}\in L_\infty(\mathbb R_+^{n+1},\mathbb C^{(n+1)^2})$, and with
$A$ being strictly accretive on a certain subspace $\mathcal{H}$ of $L_2(\mathbb R^{n},\mathbb C^{(n+1)^2})$. The key idea/discovery in these papers is that
the equation in \eqref{g1} becomes quite simple when expressing it in terms of the conormal gradient $f=\nabla_Au=[\partial_{\nu_A}u,\nabla_xu]^\ast$, $\ast$ denotes the transpose, $\partial_{\nu_A}u$ denotes the conormal derivative, instead of the potential $ u$ itself. Indeed, $f$ solves a set of generalized Cauchy-Riemann
equations expressed as a first order system
\begin{eqnarray}\label{g2}
\partial_tf+DBf=0,
\end{eqnarray}
where $D$ is a self-adjoint first order differential operator with constant coefficients and $B=B(x,t)$ is  multiplication with a bounded matrix $B$, strictly accretive on $\mathcal{H}$,  and pointwise determined by $A=A(x,t)$. The operator $DB$ is a bisectorial operator on $L_2(\mathbb R^n, \mathbb C^{(n+1)^2})$ and
if $A$, and hence $B$,  is independent of the $t$-coordinate, then  it is proved that $DB$ satisfies certain square functions estimates which implies that
$DB$, when $B$ is independent of the $t$-coordinate,  has an $L_2$-bounded holomorphic functional calculus. When $n=1$ this non-trivial fact follows from \cite{CMcM} and for $n\geq 2$ it is a consequence of the technology developed in the context of the resolution of the Kato conjecture, see
\cite{AHLMcT}, \cite{AKMc}. Using the holomorphic functional calculus for $DB$ one can then attempt to solve
\eqref{g2}, when $B$ is independent of the $t$-coordinate, by the semi-group formula $f=e^{-t|DB|}g$, with $g=g(x)$ in a suitable trace space and
$f$ has non-tangential maximal and square function estimates. The situation when $A$, and hence $B$, is dependent on the $t$-coordinate can be addressed
by perturbing the $t$-independent case and using a Picard iteration like argument, see \cite{AA}, \cite{AR}.

It is in general a very interesting program to attempt to understand to what extent the approach outlined above can be used in the context of non-linear elliptic partial differential equations and in this paper we establish one such result in the non-linear setting of operators of $p$-Laplace type. Note that there has recently been significant progress concerning the boundary behaviour of non-negative solutions
to the $p$-Laplace operator, in $\mathbb R^n$, $n\geq 1$, progress which gives at hand that many results
previous established in the linear case of the Laplace operator, $p=2$, see \cite{CFMS}, \cite{D}, \cite{JK}, remain valid also in the non-linear and potentially degenerate setting of the $p$-Laplace operator. Indeed, in \cite{LN1}, \cite{LN2}, \cite{LN3}, a number of results concerning the boundary behaviour
of positive $p$-harmonic functions, $1<p<\infty$, in a bounded Lipschitz domain $\Omega\subset\R^{n}$ were proved. In
particular, the boundary Harnack inequality and the H\"{o}lder continuity for ratios of positive $p$-harmonic functions, $1<p<\infty$, vanishing on a portion of $
\partial\Omega$ were established. Furthermore, the $p$-Martin boundary problem at $w\in
\partial\Omega$ was resolved under the assumption that $\Omega$ is either convex, $C^1$-regular or a Lipschitz domain with
small constant. Also, in \cite{LN4} these questions were resolved for $p$-harmonic functions vanishing on a portion of
certain Reifenberg flat and Ahlfors regular NTA-domains. The results and techniques developed
in \cite{LN1}-\cite{LN4} concerning $p$-harmonic functions have also been used and further
developed in \cite{LN5}, \cite{LN6} in the context of free boundary regularity in general two-phase free boundary problems
for the $p$-Laplace operator and in \cite{LN7} in the context of regularity and free boundary regularity, below the
continuous threshold, for the $p$-Laplace equation in Reifenberg flat and Ahlfors regular NTA-domains. These results are indications, and there are several others, that many results valid in the linear case may still, with the right approach, be possible to prove also in the non-linear context of the $p$-Laplace operator. While
we here restrict
ourselves to the case $n=1$, the planar case, for reasons to be discussed below, the ambition is to also understand the case $n\geq 2$ in future papers.

%
%
% Explanation of p-Laplace
%
To outline our set-up, we let $\Omega\subset\mathbb R^2$ be an unbounded domain of the form $\Omega=\{(x,y): x\in\mathbb R,\ y>\phi(x)\}$, where $\phi:\mathbb R\to\mathbb R$ denotes a Lipschitz function with constant $M$. Our main model equation is, given $1  <p < \infty$, the $p$-Laplace equation
\begin{eqnarray}\label{1.1aauu}
\divo \,(|\nabla u|^{p-2}\nabla u)=0.
\end{eqnarray}
Given  $ 1  <p < \infty, $ we denote by  $W^{1 ,p} (\Omega) $ the space of equivalence classes of functions
$f\in L^p(\Omega)$ with distributional gradients $\nabla f = ( \partial_xf, \partial_yf )$ which are
in $L^p(\Omega)$ as well.  Let $\| f \|_{1,p} = \| f \|_p +  \| \, | \nabla f | \, \|_{p}  \,  $
be the norm in $ W^{1,p}(\Omega)$ where $ \| \cdot \|_p $ denotes
the usual norm in $L^p(\Omega)$.  Next, let $ C^\infty_0 (\Omega )$ be
the set of infinitely differentiable functions with compact support in $\Omega$,
and let  $ W^{1,p}_0 ( \Omega ) $ be the closure of $ C^\infty_0 ( \Omega ) $
in the norm of $ W^{1,p} ( \Omega )$. We say that $u$ is a weak solution to \eqref{1.1aauu} in $\Omega$ provided
$u \in W^ {1,p} ( \Omega ) $  and
\begin{eqnarray}\label{1.1} \int_\Omega   |\nabla u|^{p-2}\nabla u \cdot \nabla \he   \, dxdy = 0
 \end{eqnarray}
whenever $  \he  \in W^{1, p}_0 (  \Omega )$. In the special case $p=2$ the equation in \eqref{1.1aauu} reduces to the linear Laplace equation
\eqn{generalH-}
$$
\divo (\nabla u)=\partial_{xx}u+\partial_{yy}u=0\mbox{ in }\Omega.
$$
Let $\gamma=\{(x,\phi(x)): x\in\mathbb R\}=\partial\Omega$ and consider, at a
point $(x,y)\in\gamma$, the vector fields $(0,1)$, $(1,\phi^\prime(x))$. Note that the vector field $(1,\phi^\prime(x))$ is tangential to $\gamma$ at $(x,\phi(x))$. Based on these vector fields we introduce the first order differential operators
 \begin{eqnarray}\label{a-}
\partial_\perp&:=&  (0,1)\cdot (\partial_x,\partial_y)=\partial_y,\notag\\
\partial_{||}&:=& (1,\phi^\prime(x))\cdot  (\partial_x,\partial_y)=
\partial_x+\phi^\prime(x)\partial_y=\partial_x+\phi^\prime(x)\partial_\perp.
\end{eqnarray}
Let, given $ 1  <p < \infty$, $u$ be a weak solution to \eqref{1.1aauu} in $\Omega$. Then, using interior regularity results for the $p$-Laplace operator, see \cite{DiB},
\cite{L}, \cite{T}, $u$ is $C^{1,\epsilon}$-regular locally, for some $\epsilon\in(0,1)$,
and hence $\nabla u$ is well-defined pointwise. To proceed we first fix some notation.
Here and below, we often identify $\mathbb C$ and $\mathbb R^2$, writing $a+ib =(a,b)^\ast$, where
$\ast$ denotes transpose. Sometimes we also identity $a+ib$ with the multiplication operator
$\begin{bmatrix}
a & -b \\
b& a
\end{bmatrix}$.
We parametrize $\Omega$ with
$$
y= t+\phi(x)
$$
so that $(x,y)\in\Omega$ corresponds to $(x,t)\in\mathbb R^2_+$.
We sometimes write functions $f(x,t)$ as $f_t(x)$. Now, using this notation and the operators $\partial_\perp$, $\partial_{||}$, introduced in \eqref{a-}, in Section~\ref{sec:pdetoode} we prove that  $u$ is a weak solution to \eqref{1.1aauu} in $\Omega$ if and only $$f(x,t)=(f_1(x,t), f_2(x,t))^\ast=(\partial_xu(x,y),-\partial_yu(x,y))^\ast,$$
that is $f(x,t)=\overline{ \nabla u(x, t+\phi(x))}$,
 is a solution to the first order system
\begin{eqnarray}\label{a++a++bla+-1}
\partial_{t}f+B(f)Df=0.
\end{eqnarray}
Here
$$D:=\begin{bmatrix}
0&\partial_{x} \\
-\partial_{x}&0
\end{bmatrix}
= -i \partial_{x}
$$
and
\begin{eqnarray}\label{a++a++ll1}
B(f)&=&\frac 1{\Delta_p}\begin{bmatrix}
B_{11}(f)&B_{12}(f)\\
B_{21}(f)& B_{22}(f)
\end{bmatrix},
\end{eqnarray}
with
 \begin{eqnarray}\label{a++a++llkk}
B_{11}(f)&=&(p-2)f_2^2+| f|^2,\\ B_{22}(f) &=&(p-2)f_1^2+| f|^2,\notag\\
B_{12}(f)&=&((p-2)f_1^2+| f|^2)\phi^\prime(x),\notag\\
B_{21}(f)&=&-2(p-2)f_1f_2-(\phi^\prime(x)) ((p-2)f_1^2+| f|^{ 2}),
\end{eqnarray}
and  \begin{eqnarray}\label{a++a++llkkuu}
\Delta_p&=&((p-2)f_2^2+| f|^2 )\notag\\
&&-2\phi^\prime(x)(p-2)f_1f_2+(\phi^\prime(x))^2((p-2)f_1^2+| f|^{ 2}).
\end{eqnarray}
To ease notation, here we suppress the dependence of $B(f)$ on $p$ and on $\phi'(x)$.
Note that  if $\phi^\prime\equiv 0$, that is $\Omega=\R^2_+$, then
  \begin{eqnarray}\label{a++a++ll2}
B(f)=\frac 1{ (p-2)f_2^2+| f|^2}\begin{bmatrix}
(p-2)f_2^2+| f|^2&0\\
-2(p-2)f_1f_2&(p-2)f_1^2+| f|^2
\end{bmatrix},
\end{eqnarray}
and if $p=2$, then $B(f)=B_0$ where \begin{eqnarray}\label{a++a++ll3}
B_0(x)&:=&\frac 1{1+(\phi^\prime(x))^2}\begin{bmatrix}
1&\phi^\prime(x)\\
-\phi^\prime(x)&1
\end{bmatrix}
= \frac 1{1+i\phi'(x)}.
\end{eqnarray}
In particular, when $p=2$ and  $\phi^\prime\equiv 0$, then the system in \eqref{a++a++bla+-1} reduces to the classical Cauchy-Riemann equations.

%\todo{note notation: $(Sh_t)(x)$ should rather be $(Sh)_t(x)= (Sh)(x,t)$}

Define the boundary Cauchy integral
\begin{eqnarray*}\label{uu9uuikkll}(S_0 h)_t(x)=\frac 1{2\pi i}\int_{\mathbb R}\frac {h(y)(1+i\phi^\prime(y))}{(y+i\phi(y))-(x+i(t+\phi(x)))}dy, \qquad h: \R\to \mathbb C,\end{eqnarray*}
and the solid Cauchy integral
\begin{eqnarray*}\label{uu9uuikkll}(\tilde S h)_t(x)=\frac 1{2\pi i}\iint_{\mathbb R^2_+}\frac {h(y,s)(1+i\phi^\prime(y))}{(y+i(s+\phi(y))-(x+i(t+\phi(x))}dyds,\qquad h: \R^2_+\to \mathbb C.\end{eqnarray*}
Let, given $0<\sigma <1$, $\dot H^\sigma(\R)$ denote the homogeneous fractional Sobolev space of order $\sigma$. Our main results are Theorem \ref{thm1-} and Theorem \ref{thm1} below. The first result, Theorem \ref{thm1-}, gives a Cauchy integral representation for solutions to the $p$-Laplace equation.

\begin{theorem}    \label{thm1-}
Let $1<p<\infty$, $0<\sigma <1$, and $0\le M<\infty$ be given.
Assume that $\phi:\mathbb R\to\mathbb R$ is a Lipschitz function with $\|\phi'\|_\infty\le M$
and assume that $u$ is a weak solution to \eqref{1.1aauu} in $\Omega=\{(x,y):\ x\in\mathbb R,\ y>\phi(x)\}$ satisfying
\begin{eqnarray}
\iint_\Omega|\nabla ^2u|^2(y-\phi(x))^{1-2\sigma}dxdy<\infty.
\end{eqnarray}
Let $f(x,t)=(\partial_xu(x,y),-\partial_yu(x,y))^\ast$, $y=t+\phi(x)$.
Then there exists $g\in \dot H^\sigma(\R)$ such that the Cauchy integral representation
\begin{eqnarray}\label{a++a++bla+kl}
f= S_0 g +\tilde S ((B_0-B(f)) Df),
\end{eqnarray}
holds in $\R^2_+$.
In particular, the
 trace of $f$ is $$f_0= \lim_{t\to 0^+}(S_0g) -\int_0^\infty \big( S_0(B_0-B(f_s)) Df_s\big)_{-s} ds  \in \dot H^\sigma(\R).$$
\end{theorem}

 Below we use the Cauchy integral representation in \eqref{a++a++bla+kl} to prove solvability of boundary value problems for the $p$-Laplace equation, see Theorem \ref{thm1}. In the linear case and in the end point cases corresponding to $\sigma=0$ and $\sigma=1$, such a Cauchy representation yields non-trivial trace results for elliptic equations, see \cite{AA}. However, in the case $0<\sigma<1$, which we limit ourselves to here in the non-linear case, these trace results are trivial. Indeed, as we note in Theorem~\ref{thm0} stated below, the   trace result
$$
\dot H^1(\R^2_+,t^{1-2\sigma})\to \dot H^\sigma(\R)
$$
holds in general and not only for solutions to some PDE.

\begin{theorem}\label{thm1}
Let $p, \sigma, M, \phi, $ be as in Theorem~\ref{thm1-}. Then there  exists $\delta=\delta(\sigma,M)$, $\delta>0$, such that if $|p-2|<\delta$, then the following is true.
Given any boundary data $h\in  \dot H^\sigma(\mathbb R)$, there exists
a weak solution $u$ to \eqref{1.1aauu} in $\Omega=\{(x,y):\ x\in\mathbb R,\ y>\phi(x)\}$ satisfying
\begin{eqnarray}
\iint_\Omega|\nabla ^2u|^2(y-\phi(x))^{1-2\sigma}dxdy<\infty,
\end{eqnarray}
and the boundary condition
$$
  \partial_x u(x,\phi(x))= h(x), \qquad x\in\mathbb R,
$$
where the trace of $\nabla u$ is taken in the  sense of Theorem \ref{thm0}. The same solvability result also holds true for the
boundary condition $\partial_y u(x,\phi(x))= h(x)$.
\end{theorem}

%\todo{Generalisation to $p$-Laplace type: statement of results? smallness condition on $p-2$?

%Note the very important Lemma 4.1 was missing before. We should write the extensions to more general qusilinear pdes in section 5 as you suggested, and in %particular generalize Lemma 4.1.
%}

\subsection{Organization of the paper} In section 2 we first show how quasi-linear PDEs in the plane can be reduced to  a vector valued ODE. In this section
we also show that our system of ODEs is closely related to the theory of quasiconformal and quasiregular mappings in the plane. Section 3 is devoted to
functional calculus and Cauchy type formulas in our setting and we here prove key quantitative estimates. Theorem
\ref{thm1-} and Theorem \ref{thm1} are proved in section 4, where we also for completeness include some details of the proof of Theorem \ref{thm0}. In section 5 we give a few concluding remarks discussing, in particular, generalizations of our main results to more general quasi-linear equations. We emphasize that our proofs of Theorem \ref{thm1-} and Theorem \ref{thm1} rely heavily on the fact that we are working
in the plane. For example, to be able to use the Cauchy integral representation of Theorem \ref{thm1-}
 we need to ensure that the zero sets $\{(x,t)\in\mathbb R^2_+:\ f(x,t)=0\}$ appearing in the construction, have measure zero. To conclude this we here make use of the connection to quasiregular mappings and the detailed results available concerning quasiregular mappings in the plane, see \cite{AIM},  \cite{IM},  \cite{IM1}.
Theorem \ref{thm1} is then proved by applying a fixed point argument to the Cauchy integral representation from Theorem \ref{thm1-}.

\setcounter{equation}{0} \setcounter{theorem}{0}
\section{Quasi-linear PDEs in the plane}  \label{sec:pdetoode}
\noindent
To stress generalities, in this section we consider quasi-linear equations
of the more general type
\eqn{generalHuu}
$$
\divo \,a(\nabla u)=\partial_x(a_1(\partial_xu,\partial_yu))+\partial_y(a_2(\partial_xu,\partial_yu))=0,
$$
where $a(z)=(a_1(z),a_2(z))$. Given $p$, $1<p<\infty$, we assume that the vector field $a \colon  \er^2 \to \er^2$ is $C^1$-regular and
satisfies the {growth and ellipticity assumptions
\begin{equation}\label{asp}
\left\{
    \begin{array}{c}
    |a(z)|+|\nabla a(z)||z| \leq L|z|^{p-1} \\ [5 pt]
    \n|z|^{p-2}|\xi|^{2} \leq \langle \nabla a(z)\xi, \xi
    \rangle
       \end{array}
    \right.
\end{equation}
whenever $z, \xi \in \mathbb R^2$ and for some fixed parameters  $0< \nu\leq L$. Here $\nabla a(z)$ denotes the Jacobian matrix of $a$.  We say that $u$ is a weak solution to \eqref{generalHuu} in $\Omega$ provided
$u \in W^ {1,p} ( \Omega ) $  and
\begin{eqnarray}\label{1.1} \int a(\nabla  u)\cdot\nabla \he  \, dx  = 0
 \end{eqnarray}
whenever $  \he  \in W^{1, p}_0 (  \Omega )$. If $a(z)=|z |^{ p - 2} z$,  then a solution to
\eqref{1.1} is referred to as a $p$-harmonic function and we emphasize that this main example of equations \eqref{generalHuu}, \eqref{asp}, is given by the $p$-Laplace equation introduced in \eqref{1.1aauu}.

\subsection{Reduction of the PDE to a system of ODEs}

Let $\Omega\subset\mathbb R^2$ be an unbounded domain of the form $\Omega=\{(x,y):\ x\in\mathbb R,\ y>\phi(x)\}$ where $\phi:\mathbb R\to\mathbb R$ denotes a Lipschitz function with constant $M$. Recall the  first order operators $\partial_\perp$, $\partial_{||}$, introduced in
\eqref{a-}. Using $\partial_\perp$, $\partial_{||}$ we see, given a vector field $v=(v_1,v_2)$, that
\begin{eqnarray}\label{a}
\mbox{curl } v&=&\partial_x v_2-\partial_y v_1=\partial_{||} v_2-\partial_\perp v_1-\phi^\prime(x)\partial_\perp v_2,\notag\\
\mbox{div }v&=&\partial_x v_1+\partial_y v_2=\partial_{||} v_1-\phi^\prime(x)\partial_\perp v_1+\partial_\perp v_2.
\end{eqnarray}
Let $w(x,y)=(w_1(x,y),w_2(x,y))=(\partial_xu(x,y),-\partial_yu(x,y))$, where $u$ is weak solution
to \eqref{generalHuu}. Then using \eqref{a} and \eqref{generalHuu} we have that
\begin{eqnarray}\label{a++}
0&=&-\partial_{||}  w_2-\partial_\perp  w_1+\phi^\prime(x)\partial_\perp  w_2,\notag\\
0&=&\partial_{||} a_1(\bar w)-\phi^\prime(x)\partial_\perp a_1(\bar w)+\partial_\perp a_2(\bar w).
\end{eqnarray}
Simply writing $a$ for $a(\bar w)$ and $\phi^\prime$ for $\phi^\prime(x)$, we see that the second relation \eqref{a++} can be expressed as
\begin{eqnarray}\label{a++a}
0&=&(\partial_{1}a_1)\partial_{||} w_1-(\partial_{2}a_1)\partial_{||} w_2 -\phi^\prime\bigl((\partial_{1}a_1)\partial_{\perp} w_1-(\partial_{2}a_1)\partial_{\perp} w_2\bigr )\notag\\
&&+(\partial_{1}a_2)\partial_{\perp} w_1-(\partial_{2}a_2)\partial_{\perp} w_2.
\end{eqnarray}
We next want to solve for $(\partial_{\perp} w_1,\partial_{\perp} w_2)$ in the system
\begin{eqnarray}\label{a++a+}
0&=&-\partial_{||}  w_2-\partial_\perp  w_1+\phi^\prime\partial_\perp  w_2,\notag\\
0&=&(\partial_{1}a_1)\partial_{||} w_1-(\partial_{2}a_1)\partial_{||} w_2 -\phi^\prime(x)\bigl((\partial_{1}a_1)\partial_{\perp} w_1-(\partial_{2}a_1)\partial_{\perp} w_2\bigr )\notag\\
&&+(\partial_{1}a_2)\partial_{\perp} w_1-(\partial_{2}a_2)\partial_{\perp} w_2.
\end{eqnarray}
Let \begin{eqnarray*}\label{a++a++}
A:=\begin{bmatrix}
-1& \phi^\prime\\
(\partial_{1}a_2)-\phi^\prime(\partial_{1}a_1)&-(\partial_{2}a_2)+\phi^\prime(\partial_{2}a_1)
\end{bmatrix},\ D:=\begin{bmatrix}
0&\partial_{||} \\
-\partial_{||}&0
\end{bmatrix}.
\end{eqnarray*}
Using this notation, the system in \eqref{a++a+} can be written as
\begin{eqnarray}\label{a++a++}
A\begin{bmatrix}
\partial_{\perp} w_1\\
\partial_{\perp} w_2
\end{bmatrix}= \begin{bmatrix}
1&0\\
\partial_{2}a_1&\partial_{1}a_1
\end{bmatrix} D\begin{bmatrix}
 w_1\\
 w_2
\end{bmatrix}=\begin{bmatrix}
1&0\\
\partial_{2}a_1&\partial_{1}a_1
\end{bmatrix} \begin{bmatrix}
 \partial_{||}w_2\\
 -\partial_{||}w_1
\end{bmatrix}.
\end{eqnarray}
In the following, we let
\begin{eqnarray}\label{a++a+++}
\Delta&:=&-((\partial_{1}a_2)-\phi^\prime(\partial_{1}a_1))\phi^\prime+((\partial_{2}a_2)-\phi^\prime(\partial_{2}a_1))\notag\\
&=&(\partial_{2}a_2)-\phi^\prime((\partial_{1}a_2)+(\partial_{2}a_1))+(\phi^\prime)^2(\partial_{1}a_1).
\end{eqnarray}
Using this, we have
\begin{eqnarray*}\label{a++a++ll}
A^{-1}=-\frac 1\Delta\begin{bmatrix}
(\partial_{2}a_2)-\phi^\prime(\partial_{2}a_1)&\phi^\prime\\
(\partial_{1}a_2)-\phi^\prime(\partial_{1}a_1)&1
\end{bmatrix},
\end{eqnarray*}
and
\begin{eqnarray*}\label{a++a++pp}
A^{-1}\begin{bmatrix}
1&0\\
(\partial_{2}a_1)&(\partial_{1}a_1)
\end{bmatrix}=-\frac 1{ \Delta}\begin{bmatrix}
\partial_{2}a_2&(\partial_{1}a_1)\phi^\prime\\
(\partial_{2}a_1)+(\partial_{1}a_2)-\phi^\prime(\partial_{1}a_1)&\partial_{1}a_1
\end{bmatrix}.
\end{eqnarray*}
Let, for $ w=( w_1, w_2)$ and $\phi$ given,
\begin{eqnarray*}\label{a++a++b}
 B^{w,\phi}(x,t):=\frac 1{ \Delta}\begin{bmatrix}
(\partial_{2}a_2)(\bar w)&(\partial_{1}a_1)(\bar w)\phi^\prime(x)\\
((\partial_{2}a_1)(\bar w)+(\partial_{1}a_2)(\bar w)-\phi^\prime(x)(\partial_{1}a_1)(\bar w))&(\partial_{1}a_1)(\bar w)
\end{bmatrix}.
\end{eqnarray*}
Then \eqref{a++a++} can be restated as
\begin{eqnarray}\label{a++a++blalu}
\begin{bmatrix}
\partial_{\perp} w_1\\
\partial_{\perp} w_2
\end{bmatrix}+ B^{w,\phi} D\begin{bmatrix}
 w_1\\
 w_2
\end{bmatrix}=0.
\end{eqnarray}
We summarize our findings as follows.
\begin{lemma}\label{reduk}  A function $u$ is a weak solution to \eqref{generalHuu} in $\Omega$ if and only if
$$f(x,t)=(f_1(x,t), f_2(x,t)):=\big(\partial_xu(x,t+\phi(x)),-\partial_yu(x,t+\phi(x))\big)^\ast$$ satisfies
\begin{eqnarray}\label{a++a++bla+lu}
\partial_{t}\begin{bmatrix}
f_1\\
f_2
\end{bmatrix}+B^{f,\phi}D\begin{bmatrix}
 f_1\\
f_2
\end{bmatrix}=0,\ D:=\begin{bmatrix}
0&\partial_{x} \\
-\partial_{x}&0
\end{bmatrix},
\end{eqnarray}
in $\mathbb R_+^2:=\{(x,t)\in\mathbb R^2:\ t>0\}$ where
\begin{multline}\label{a++a++ll}
B^{f,\phi}(x,t) \\:=
\frac 1{ \Delta}\begin{bmatrix}
\partial_{2}a_2(\bar f)&(\partial_{1}a_1(\bar f))\phi^\prime(x)\\
(\partial_{2}a_1(\bar f))+(\partial_{1}a_2(\bar f))-\phi^\prime(x)(\partial_{1}a_1(\bar f))&\partial_{1}a_1(\bar f)
\end{bmatrix}
\end{multline}
and
\begin{multline}\label{a++a++ll+}
\Delta :=(\partial_{2}a_2)(\bar f)-\phi^\prime(x)\big((\partial_{1}a_2)(\bar f)+(\partial_{2}a_1)(\bar f)\big)+(\phi^\prime(x))^2(\partial_{1}a_1)(\bar f).
\end{multline}
\end{lemma}

Recall that a $2\times 2$-dimensional matrix $B$, defined in $\mathbb R^2$ and potentially complex valued, is said to be accretive if
\begin{eqnarray}\label{uu3}
\kappa:=\mbox{essinf}_{(x,t)\in \mathbb R^{2}}\inf_{\xi\in \mathbb C^{2}\setminus\{0\}}\frac{\mbox{Re}(B(x,t)\xi,\xi)}{|\xi|^2}>0.
\end{eqnarray}
\begin{lemma}\label{reduk+} Let $a \colon  \er^2 \to \er^2$  be a $C^1$-regular vector field satisfying \eqref{asp} for  some fixed parameters  $0< \nu\leq L$. Let $B=B^{f,\phi}$ be as in \eqref{a++a++ll}, \eqref{a++a++ll+}. Then $B\in L_\infty(\mathbb R^{2}_+,\mathbb C^2)$ and $B$
is accretive in the sense of \eqref{uu3}. Furthermore, the $L_\infty$-bound on $B$, and the parameter of accretivity $\kappa$, depend only on $p$, $M$,  $\nu$, and $L$.
\end{lemma}
\begin{proof} First, using the ellipticity type condition in \eqref{asp} we see that $$\Delta\geq\nu|\bar f|^{p-2}(1+(\phi^\prime(x))^2).$$ Hence, using also the upper bound in
\eqref{asp} we can conclude that
$$|B|\leq\frac {c(L)} \Delta |\bar f|^{p-2}(1+\phi^\prime(x))\leq c(p,M,\nu,L).$$
To estimate the parameter of accretivity, let $\xi\in \mathbb C^{2}\setminus\{0\}$, $(x,t)\in \mathbb R^{2}_+$, and note that
\begin{eqnarray}\label{uu3ll}
\mbox{Re}(B(x,t)\xi,\xi)&=&B_{11}|\xi_1|^2+B_{22}|\xi_2|^2+B_{12}\mbox{Re}(\bar\xi_1\xi_2)+B_{21}\mbox{Re}(\xi_1\bar\xi_2)\notag\\
&=&(\partial_{1}a_1(\bar f))|\xi_1|^2+(\partial_{2}a_2(\bar f))|\xi_2|^2\notag\\
&+&((\partial_{1}a_2(\bar f))+(\partial_{2}a_1(\bar f)))\mbox{Re}(\xi_1\bar\xi_2)
\end{eqnarray}
and the estimate now follows from \eqref{asp}.
\end{proof}

\subsection{Quasi-regular mappings in the plane}

Consider a function  $f:\Omega\to\Omega'$ where $\Omega,\Omega'\subset\mathbb C$. Let
$z=x+iy\in\mathbb C$ and assume that $f$ has a derivative $\nabla f$ at $z$. We let
$\partial_zf=(\partial_xf-i\partial_yf)/2$, $\partial_{\bar z}f=(\partial_xf+i\partial_yf)/2$ and we write the derivative as
$$\nabla f(z)h=\partial_zf(z)h+\partial_{\bar z}f(z)\bar h,\ h\in \mathbb C=\mathbb R^2.$$ Note that
$$|\nabla f(z)|^2=|\partial_zf(z)|^2+|\partial_{\bar z}f(z)|^2$$
and that the Jacobian equals
$$J(z,f)=|\partial_zf(z)|^2-|\partial_{\bar z}f(z)|^2.$$ Recall that if the mapping $f$ satisfies $f\in W^{1,2}_{{loc}}(\Omega)$,
 $f$ is orientation preserving so that $J(z,f)\geq 0$ a.e., and if
\begin{eqnarray}\label{qr}
|\nabla f(z)|^2\leq K J(z,f)\mbox{ for almost every $z\in\Omega$},
\end{eqnarray}
then $f$ is called $K$-quasiregular.
The smallest number $K$ for which \eqref{qr} holds is called the dilation of $f$ and we denote this number by $K(f)$. Constant functions are by definition $0$-quasiregular. If, in addition, $f$ is a homeomorphism, then $f$ is called $K$-quasiconformal. Note that \eqref{qr} can also be expressed as
\begin{eqnarray}\label{qr+}
|\partial_zf(z)|^2+|\partial_{\bar z}f(z)|^2\leq K(|\partial_zf(z)|^2-|\partial_{\bar z}f(z)|^2)
\end{eqnarray}
or equivalently
\begin{eqnarray}\label{qr++}
|\partial_{\bar z}f(z)|^2\leq \frac {K-1}{K+1}|\partial_zf(z)|^2\mbox{ for almost every $z\in\Omega$}.
\end{eqnarray}
In particular, $f:\Omega\to\Omega'$ is $K$-quasiregular if and only if  $f\in W^{1,2}_{{loc}}(\Omega)$,  $f$ is orientation preserving and \begin{eqnarray}\label{qr+++}
\partial_{\bar z}f(z)=\mu(z)\partial_zf(z)\mbox{ for almost every $z\in\Omega$,}
\end{eqnarray}
where $\mu$, called the Beltrami coefficent of $f$, is a bounded measurable function satisfying
\begin{eqnarray}\label{qr++++}
\|\mu\|_\infty\leq \sqrt{\frac {K-1}{K+1}}<1.
\end{eqnarray}
Note that the differential equation in \eqref{qr+++} is called the Beltrami equation and it is this equation that provides the link from the geometric theory of quasiconformal mappings to complex analysis and to elliptic partial differential equations. For accounts of these connections we refer to  \cite{AIM},  \cite{IM},  and \cite{IM1}. The following lemma connects
the notion of quasiregular mappings to the set-up used in this paper.
\begin{lemma}\label{rem1uu}
Let $f = (f_1 , f_2 )\in W^{1,2}_{{loc}}(\Omega)$ be non-constant. If $\partial_tf+BDf=0$
for some bounded and accretive $B$ then $f$ is quasiregular. Furthermore, if $f$ is quasiregular and if we define the complex linear multiplier $B:=-\partial_tf/Df$,  then $B$ is  bounded and accretive and
$\partial_tf+BDf=0$.
\end{lemma}
\begin{proof} Assume that $\partial_tf+BDf=0$ for some bounded and accretive $B$. Simply note that
\begin{eqnarray}\label{qr++++pp}
|\nabla f|^2&=&|\partial_xf_1|^2+|\partial_xf_2|^2+|\partial_tf_1|^2+|\partial_tf_2|^2=|\partial_tf|^2+|Df|^2,\notag\\
J(f)&=&\partial_xf_1\partial_tf_2-\partial_xf_2\partial_tf_1=-(\partial_tf,Df),
\end{eqnarray}
and hence $|\nabla f|^2\approx |Df|^2\lesssim (BDf, Df)= -(\partial_t f, Df)$, so $f$ is quasiregular. To prove
the other direction, assume that $f$ is quasiregular and let $B:=-\partial_tf/Df$ by complex division. Then $\partial_tf+BDf=0$ and $B$ is bounded since $|\partial_t f|^2+ |Df|^2 \lesssim |\partial_t f| |Df|$. Moreover
\begin{eqnarray}\label{qr++++pp1}
1\approx \frac {J(f)}{|\nabla f|^2}=\frac{-(\partial_t f, Df)}{|\partial_t f|^2+|Df|^2}\approx \frac{(BDf, Df)}{|Df|^2}= \mbox{Re} (B),
\end{eqnarray}
so $B$ is accretive. This completes the proof of the lemma.
\end{proof}

We next note the following existence and uniqueness result for the Beltrami equation in \eqref{qr+++}, assuming that $\mu$ has compact support, as well as
the Stoilow factorization of quasiregular mappings with subsequent corollary.
Besides the more modern references given below for these results, we also refer the reader to the very readable lecture notes \cite[Chapter V]{Ahl}.

\begin{theorem}\label{exU} Let $\mu$ bounded measurable function on $\mathbb C$ with compact support and assume that
\begin{eqnarray}\label{qr++++ll}
\|\mu\|_\infty\leq k\mbox{ for some $k<1$}.
\end{eqnarray} Then there exists a
unique $f\in W^{1,2}_{{loc}}(\Omega)$ such that
\begin{eqnarray}\label{qr+++ll}
\partial_{\bar z}f(z)&=&\mu(z)\partial_zf(z)\mbox{ for almost every $z\in \mathbb C$},\notag\\
f(z)&=&z+O(z^{-1})\mbox{ as $z\to\infty$}.
\end{eqnarray}
Moreover, there exists $p(k)$ such that $f\in W^{1,p}_{{loc}}(\Omega)$ for all $p$, $2\leq p<p(k)$.
\end{theorem}
\begin{proof} This is Theorem 5.1.2 in \cite{AIM}.
\end{proof}
\begin{theorem}\label{stoilow} Let $f:\Omega\to\Omega'$ be a homeomorphic solution to the Beltrami equation in
\eqref{qr+++}, with $|\mu(z)|\leq k<1$ almost everywhere on $\Omega$, and assume that $f\in W^{1,1}_{{loc}}(\Omega)$.
Suppose that
$g\in W^{1,2}_{{loc}}(\Omega)$ is any other solutions to \eqref{qr+++}. Then there exists a holomorphic function $\Phi:\Omega'\to\mathbb C$ such that
$$g(z)=\Phi(f(z)),\ z\in\Omega.$$
Conversely, if $\Phi$ is holomorphic in $\Omega'$, then the composition $\Phi\circ f$ is a $W^{1,2}_{{loc}}$-solution to
\eqref{qr+++} in $\Omega$.
\end{theorem}
\begin{proof} This is Theorem 5.5.1 in \cite{AIM}.
\end{proof}
\begin{corollary}\label{lemmata1+} Let $f$ be a non-constant quasiregular mapping defined on a domain $\Omega\subset\mathbb C$. Then
\begin{enumerate}
\item $f$ is open and discrete,
\item $f$ is locally H{\"o}lder continuous with exponent $\alpha=1/K$, $K=K(f)$, and
\item $f$ is differentiable with non-vanishing Jacobian almost everywhere.
\end{enumerate}
\end{corollary}
\begin{proof} This is essentially Corollary 5.5.2 in \cite{AIM}.
\end{proof}

Recall that a mapping $f:\Omega \, \rar \, \mathbb R^2 $ is discrete if $f^{-1}(y)$ is a discrete set for all $y\in \mathbb R^2$, and $f$ is open if it takes open sets onto open sets. That $f^{-1}(y)$ is a discrete set means that it is made up by  isolated points. We also note the following lemma concerning the convergence of  $K$-quasiregular mappings.
\begin{lemma}\label{th2}
Let $f_j:\Omega\to\mathbb R^2$, $j=1,..$, be a sequence of $K$-quasiregular mappings converging locally uniformly to a mapping $f$. Then $f$ is quasiregular
and
$$K(f)\leq\limsup_{j\to \infty}K(f_j).$$
\end{lemma}
\begin{proof} See, for example Theorem 8.6 in \cite{Ri} and the discussion above Theorem 2.4 in the same reference.
\end{proof}

Beltrami equations can be reduced
to real elliptic divergence form equations and the following lemma can be verified by a straightforward calculation.
\begin{lemma}\label{reg}
Let $f = (f_1 , f_2 )\in W^{1,2}_{{loc}}(\Omega)$ satisfy \eqref{qr+++} for some $\mu\in L_\infty(\Omega,\mathbb C)$ satisfying
\eqref{qr++++}. Define a $2\times 2$-matrix $A=A_\mu=\{a_{ij}\}$ as follows. Given $\mu=\mu_1+i\mu_2$, $\mu_i=\mu_i(z)\in\mathbb R$, $z\in\mathbb C$, we let
\begin{eqnarray}
a_{11}(z)&=&\frac{1-2\mu_1(z)+|\mu(z)|^2}{1-|\mu(z)|^2},\ a_{22}(z)=\frac{1+2\mu_1(z)+|\mu(z)|^2}{1-|\mu(z)|^2},\notag\\
a_{12}(z)&=&a_{21}(z)=-\frac {2\mu_2(z)}{1-|\mu(z)|^2}.
\end{eqnarray}
Then $A$ is bounded, symmetric and satisfies $\mbox{det } A(z)=1$ for a.e. $z\in \Omega$. Furthermore, $f_1$ and $f_2$ are weak solutions to the equation
$$\text{div}(A\nabla \cdot)=0\mbox{ in $\Omega$}.$$
\end{lemma}

\begin{remark}\label{rem1-}
Consider the matrix $A$ in the statement of Lemma \ref{reg}. Using that $\mbox{det } A=1$ one easily see that the eigenvalues of
$A(z)$ are
$$\lambda_\pm(z):=\frac {1+|\mu(z)|^2}{1-|\mu(z)|^2}\pm\sqrt{\frac {1+|\mu(z)|^2}{1-|\mu(z)|^2}-1}.$$
Since $|\mu(z)|<1$ we immediately see that $\lambda_\pm(z)$ are greater or equal to 1 and that
\begin{eqnarray}
\sup_{z\in\Omega}\lambda_-(z)\leq\sup_{z\in\Omega}\lambda_+(z) \leq\frac {1+\beta^2}{1-\beta^2}+\sqrt{\frac {1+\beta^2}{1-\beta^2}-1}
\end{eqnarray}
if $|\mu(z)|<\beta$ for all $z\in\Omega$. In particular, if this is the case then $A$ is uniformly elliptic. Naturally an upper bound can also be derived by simply using the explicit expression of the coefficients $\{a_{ij}\}$.
\end{remark}

%\todo{isn't Lemma 2.6 redundant after Corollary~\ref{lemmata1+} (2)? If you decide to delete it, please change references to it below.}
The following lemma is essentially statement (2) in Corollary~\ref{lemmata1+} but we include it, and a short discussion of its proof based on PDE-techniques, to stress the connection between
the Beltrami equation and quasi-linear PDEs.
\begin{lemma}\label{rem1}
Let $f=(f_1 , f_2 )$ be as in the statement of Lemma \ref{reg} and assume that $|\mu(z)|\leq\beta<1$ on $\Omega$. Then there exist
 $c=c(\beta)$, $1\leq c<\infty$, and $\sigma=\sigma(\beta)\in(0,1),$ such that if
 $B(z,2R)\subset\Omega$ then
 $$\sup_{z_1,z_2\in B(z,r)}|f(z_2)-f(z_1)| \leq c(r/R)^\sigma\biggl (R^{-2}\int_{B(z,2R)}|f|^2dz\biggr )^{1/2}.$$
\end{lemma}
\begin{proof}
Let $f=(f_1 , f_2 )$ be as in the statement of Lemma \ref{reg} and assume that $|\mu(z)|\leq\beta<1$ on $\Omega$. Consider a ball $B(z,R)$ such that $B(z,2R)\subset\Omega$. Then, using Lemma \ref{reg} and Moser iteration we have that
$$\sup_{z_1,z_2\in B(z,r)}|f(z_2)-f(z_1)| \leq c(r/R)^\sigma\biggl (R^{-2}\int_{B(z,2R)}|f|^2dz\biggr )^{1/2}$$
where $c$, $\sigma\in(0,1)$ are independent of $f$, $r$ and $R$. In fact, $c$ and $\sigma$ only depend on
the the operator through the ellipticity and the bounded on the coefficients and if $|\mu(z)|\leq\beta<1$ on $\Omega$, see Lemma \ref{rem1-}, then $c$ and $\sigma$ will depend on $\mu$ through $\beta$.
\end{proof}

 \setcounter{equation}{0} \setcounter{theorem}{0}
 \section{Functional calculus and Cauchy operators}\label{fcal}
 \noindent
 Recall the definition of $D$ introduced below
 \eqref{a++a++bla+-1} and the matrix $B_0$  defined in \eqref{a++a++ll3}. Given $p$, $f$, $\phi$, we in the following  write
 $B(f)=B_p^{f,\phi}$ for this generic $t$-dependent matrix.  In line with \cite{AA}, \cite{AR} we approach the system in \eqref{a++a++bla+-1} using a functional calculus build on the $t$-independent
 matrix $B_0$. As references for functional calculus we refer to \cite{ADMc}, \cite{Ar}, \cite{AMcN}, \cite{DS}, \cite{H}, \cite{Mc}, \cite{Mc1}, \cite{McY}. Let
$L_2(\mathbb R)=L_2(\mathbb R,\mathbb C)$ and let, for $-1\leq\alpha\leq 1$,
\begin{eqnarray}\label{uu4}
L_2(\mathbb R_+^{2},t^\alpha):=\biggl \{f:\mathbb R^{2}_+\to\mathbb C^{2}:\ \int\int_{\mathbb R^{2}_+}|f(x,t)|^2t^\alpha dtdx<\infty\biggr \}.
\end{eqnarray}
Then, both as an operator in $L_2(\mathbb R)$ and in $L_2(\mathbb R_+^{2},t^\alpha)$, acting in the $x$-variable
for each fixed $t>0$, $DB_0$ and $B_0D$ define closed and densely defined operators with spectrum contained in a bisector $S_\omega=S_{\omega^+}\cup (-S_{\omega^+})$ where
\begin{eqnarray}\label{uu8uu}
S_{\omega^+}:=\{\lambda\in\mathbb C:\ |\mbox{arg }\lambda|\leq \omega\}\cup\{0\},\ \omega<\pi/2.
\end{eqnarray}
 In particular, as a consequence of \cite{CMcM} we note that both $DB_0$ and $B_0D$ have bounded holomorphic functional calculi which supply estimates of operators $\psi(DB_0)$ and $\psi(B_0D)$ formed by applying
holomorphic functions $\psi:S^o_\mu\to\mathbb C$, $\omega<\mu$, to the operators $DB_0$ and $B_0D$ respectively.
Here $S_\mu^o= S^o_{\mu^+}\cup (-S^o_{\mu^+})$ denotes the open bisector, where
$$
  S^o_{\mu^+}:= \{\lambda\in\mathbb C:\ |\mbox{arg }\lambda|< \mu\}\setminus\{0\},
$$
Note here, in particular, that $\psi$ need not be analytic across $0$ or at $\infty$.

Applying the functional calculi
with the scalar holomorphic functions
$\lambda\to|\lambda|:=\pm\lambda$, if $\pm\mbox{Re }\lambda>0$, $\lambda\to e^{-t|\lambda|}$, $\lambda\to \chi_{\pm}(\lambda):=1$ if
$\pm\mbox{Re }\lambda>0$ and 0 elsewhere, 
%here $\mbox{Re }\lambda$ is the real part of $\lambda$, 
we get operators
\begin{eqnarray}\label{uu9uull}
&&\Lambda_0:=|DB_0|,\ e^{-t|DB_0|}, t>0,\ E_0^\pm:=\chi_{\pm}(DB_0),\notag\\
&&\tilde\Lambda_0:=|B_0D|,\ e^{-t|B_0D|}, t>0,\ \tilde E_0^\pm:=\chi_{\pm}(B_0D),
\end{eqnarray}
acting as operators in $L_2(\mathbb R)$.
We note that the operators $\tilde E_0^\pm$ are projections and that we have a topological splitting
\begin{eqnarray}\label{uu9uull1}
L_2(\mathbb R)=\tilde E^+_0L_2(\mathbb R)\oplus \tilde E^-_0L_2(\mathbb R).
\end{eqnarray}
Using this notation we define the operators
$$
  (S_0 h)_t(x)=
  \begin{cases}
    (e^{-t\tilde\Lambda_0}\tilde E_0^{+}h)(x), &\qquad t>0, \\
    -(e^{t\tilde \Lambda_0}\tilde E_0^{-}h)(x), & \qquad t<0,
  \end{cases}
$$
and
\begin{eqnarray}  \label{uu9uull1+bb}
(\tilde Sh)_t(x)&:=&\int_0^te^{-(t-s)\tilde \Lambda_0}\tilde E_0^+B_0h_s(x)ds\notag\\
&&+\int_t^\infty e^{-(s-t)\tilde\Lambda_0}\tilde E_0^-B_0h_s(x)ds,\qquad t>0.
\end{eqnarray}
As we will see in Lemma \ref{CB} below, these operators coincide with the boundary and solid Cauchy integrals from the introduction.

We intend to derive a representation formula for solutions to the equation
\begin{eqnarray}\label{uu9uull1+a}
\partial_tf_t+B(f_t)Df_t=0,\ f_t\in L_2(\mathbb R).
\end{eqnarray}
To do this we first write
\begin{eqnarray}\label{uu9uull1++}
&&\partial_tf_t+B_0Df_t=B_0\mathcal{E}_tDf_t,\mbox{ where }\notag\\
&&\mathcal{E}_t=\mathcal{E}_t(x):=(\mathcal{E}(f))_t(x):=I-(B_0(x))^{-1}B(f(x,t)).
\end{eqnarray}
Applying the projections $\tilde E_0^\pm$ we see that
\begin{eqnarray}\label{uu9uull1+kk}
&&\partial_tf_t^++\tilde\Lambda_0f_t^+=\tilde E_0^+B_0\mathcal{E}_tDf_t,\notag\\
&&\partial_tf_t^--\tilde\Lambda_0f_t^-=-\tilde E_0^-B_0\mathcal{E}_tDf_t.
\end{eqnarray}
where now $f_t^\pm=\tilde E_0^\pm f_t$. Formally assuming that $$\lim_{t\to 0^+}f_t=f_0,\ \lim_{t\to\infty}f_t=0,$$ we can integrate the equations in
\eqref{uu9uull1+kk} to conclude that
\begin{eqnarray}\label{uu9uull1+}
f_t^+-e^{-t\tilde \Lambda_0}f_0^+&=&\int_0^te^{-(t-s)\tilde \Lambda_0}\tilde E_0^+B_0\mathcal{E}_sDf_sds,\notag\\
0-f_t^-&=&-\int_t^\infty e^{-(s-t)\tilde \Lambda_0}\tilde E_0^-B_0\mathcal{E}_sDf_sds.
\end{eqnarray}
Subtracting the equations in \eqref{uu9uull1+}, we see that
\begin{eqnarray}\label{uu9uull1+hha}
f_t= (S_0f_0)_t+\tilde S\mathcal{E}_tDf_t,
\end{eqnarray}
and we have derived a representation formula for \eqref{uu9uull1+a}. To continue we note that
\begin{eqnarray}\label{uu9uull1+l}
Df_t&=&De^{-t\tilde \Lambda_0}\tilde E_0^+f_0+D\tilde S\mathcal{E}_tDf_t.
\end{eqnarray}
However, using the relation $D\psi(B_0D)=\psi(DB_0)D$ through the holomorphic functional calculus we first see that
$De^{-t\tilde \Lambda_0}\tilde E_0^+f_0=e^{-t\Lambda_0}E_0^+Df_0$. Furthermore, by the same argument we have that
\begin{eqnarray}\label{uu9uull1+ll}
&&De^{-(t-s)\tilde \Lambda_0}\tilde E_0^+B_0=\Lambda_0e^{-(t-s)\Lambda_0}E_0^+,\notag\\
&&De^{-(s-t)\tilde\Lambda_0}\tilde E_0^-B_0=\Lambda_0e^{-(t-s)\Lambda_0}E_0^-.
\end{eqnarray}
Define
\begin{eqnarray}\label{uu9uull1+pka}
(Sh)_t(x):=D(\tilde Sh)(x)&=&\int_0^t\Lambda_0 e^{-(t-s)\Lambda_0}E_0^+h_s(x)ds\notag\\
&&+\int_t^\infty \Lambda_0 e^{-(s-t)\Lambda_0}E_0^-h_s(x)ds,
\end{eqnarray}
for $h\in L_2(\mathbb R)$ so that
\begin{eqnarray}\label{uu9uull1+p}
Df_t=e^{-t\Lambda_0}E_0^+Df_0+S\mathcal{E}_tDf_t.
\end{eqnarray}
For a rigorous definition of the singular integral operator $S$, see \cite{AA}. We here note the following lemma concerning the operators  $S_0$, $\tilde S$ and $S$ and their relation to classical analysis in the complex plane.
\begin{lemma}\label{CB} Let $S_0$,  $\tilde S$  and $S$ be defined as above. Then $S_0$ is the boundary Cauchy integral
$$(S_0h)_t(x)=\frac 1{2\pi i}\int_{\mathbb R}\frac {h(y)(1+i\gamma^\prime(y))dy}{(y+i\gamma(y))-(x+i\gamma(x))},$$
$\tilde S$ is the solid Cauchy integral
$$(\tilde S h)_t(x)=\frac 1{2\pi i}\iint_{\mathbb R^2_+}\frac {h(y,s)(1+i\phi^\prime(y))}{(y+i(s+\phi(y))-(x+i(t+\phi(x))}dyds,$$
and $S=D\tilde S$ is the Beurling transform
$$(S h)_t(x)=-\frac 1{2\pi}\mbox{ p.v. }\iint_{\mathbb R^2_+}\frac {h(y,s)(1+i\phi^\prime(y))}{((y+i(s+\phi(y))-(x+i(t+\phi(x)))^2}dyds.$$
\end{lemma}
\begin{proof}
  The result for the boundary Cauchy singular integral is well known, see for example~\cite{McQ}.
  Integration and derivation gives the result for $\tilde S$ and $S$.
\end{proof}

To further understand the appropriate function spaces, we consider the free evolution
$e^{-t\tilde \Lambda_0}\tilde E_0^+f$ and we note that
\begin{eqnarray}\label{si1}
&&\iint_{\mathbb R_+^2}|De^{-t\tilde\Lambda_0}\tilde E_0^+f|^2t^{1-2\sigma}dxdt\notag\\
&\approx& \iint_{\mathbb R_+^2}|B_0De^{-tB_0D}\tilde E_0^+f|^2t^{1-2\sigma}dxdt\notag\\
&=&\int_{\mathbb R_+}\biggl (\int_{\mathbb R}\biggl|t^{1-\sigma}B_0De^{-t\tilde\Lambda_0}\tilde E_0^+f\biggr |^2dx\biggr )\frac {dt}t.
\end{eqnarray}
Note that $\tilde\Lambda_0=\mbox{sgn}(B_0 D)B_0 D$ and since $\mbox{sgn}( B_0 D)$ is invertible we can conclude that
\begin{eqnarray}\label{si2}
&&\iint_{\mathbb R_+^2}|De^{-t\tilde\Lambda_0}\tilde E_0^+f|^2t^{1-2\sigma}dxdt\notag\\
&\approx&\int_{\mathbb R_+}\biggl (\int_{\mathbb R}\biggl|(t\tilde\Lambda_0)^{1-\sigma}e^{-t\tilde\Lambda_0}\tilde\Lambda_0^\sigma\tilde E_0^+f\biggr |^2dx\biggr )\frac {dt}t.
\end{eqnarray}
Using that $\psi(tB_0 D):=(t\tilde\Lambda_0)^{1-\sigma}e^{-tB_0 D}$ satisfies square function estimates when $\sigma<1$, it follows from \eqref{si2} that
\begin{eqnarray}\label{si3}
\iint_{\mathbb R_+^2}|De^{-t\tilde\Lambda_0}\tilde E_0^+f|^2t^{1-2\sigma}dxdt\approx
\int_{\mathbb R}\biggl|\tilde\Lambda_0^\sigma\tilde E_0^+f\biggr |^2dx.
\end{eqnarray}
However, $\mathcal{D}(\tilde\Lambda_0)=\mathcal{D}(B_0 D)\approx \mathcal{D}(D)\approx \mathcal{D}(\nabla)=\dot H^1(\mathbb R)$
and $\mathcal{D}(\tilde\Lambda_0^0)=L_2(\mathbb R)$. Hence, as in \cite{R}, by interpolation we see that
\begin{eqnarray}\label{sys2aa-6}
\mathcal{D}(\tilde\Lambda_0^{\sigma})=\dot H^{\sigma}(\mathbb R).
\end{eqnarray}
In particular, we have the following estimates.
\begin{lemma}\label{lem1} Let $0\leq\sigma\leq 1$. For all $f\in \mathcal{D}(\tilde\Lambda_0^{\sigma})$ we have
\begin{eqnarray*}\label{uu17}
\|\tilde\Lambda_0^{\sigma}f\|_{2}\approx \|f\|_{\dot{H}^{\sigma}(\mathbb R)}.
\end{eqnarray*}
\end{lemma}
\begin{lemma}\label{lemmata1} The following estimates holds.
\begin{enumerate}
\item For $\sigma\in [0,1]$,
$$\sup_{t>0}\|e^{-t\tilde\Lambda_0}\tilde E_0^+f\|_{\dot H^\sigma(\mathbb R)}\leq c\|f\|_{\dot H^\sigma(\mathbb R)}.$$
\item For $\sigma\in [0,1)$,
$$\iint_{\mathbb R_+^2}|De^{-t\tilde\Lambda_0}\tilde E_0^+f|^2t^{1-2\sigma}dxdt\leq c\|f\|^2_{\dot H^\sigma(\mathbb R)}.$$
\end{enumerate}
\end{lemma}

\begin{remark}
  Below we refer to \cite{R} for many results concerning solving linear boundary value problems
  for first order systems of the form
  $$
    \partial_t f+ B_0 Df=0,
  $$
  with boundary trace $f|_{\R}\in \dot H^{\sigma}(\R)$.
  Applying the isomorphism $D: \dot H^{\sigma}(\R)\to \dot H^{\sigma-1}(\R)$, pointwise in $t$, to this equation,
   yields the system
   $$
     \partial_t \tilde f+ DB_0 \tilde f=0,
   $$
   for $\tilde f:= Df\in \dot H^{\sigma-1}(\R)$. Thus the results stated in \cite{R} for $\tilde f$, transfer directly to
   the setting in this paper for $f$.
\end{remark}

To establish estimates for the operators $S$ and $\tilde S $ we first note that for a multiplier $\mathcal{E}$ we have
\begin{eqnarray}\label{uu26}
\sup_{\|f\|_{L_2(\mathbb R_+^{2},t^{1-2\sigma})}=1}\|\mathcal{E}f\|_{L_2(\mathbb R_+^{2},t^{1-2\sigma})}=\|\mathcal{E}\|_{L_\infty(\mathbb R_+^{2})}.
\end{eqnarray}
We next prove the following lemma.
\begin{lemma}\label{lemmata2} The following estimates holds.
\begin{enumerate}
\item For $\sigma\in [0,1/2)$,
$$\sup_{t>0}\|\tilde S \mathcal{E}_tDf_t\|^2_{\dot H^\sigma(\mathbb R)}\leq c\|\mathcal{E}\|^2_{L_\infty(\mathbb R_+^{2})}\iint_{\mathbb R_+^2}|Df(x,t)|^2t^{1-2\sigma}dxdt.$$
\item For $\sigma\in [1/2,1)$,
$$\sup_{t>0}\frac 1 t\int_t^{2t}\|\tilde S\mathcal{E}_sDf_s\|_{\dot H^\sigma(\mathbb R)}ds\leq c\|\mathcal{E}\|^2_{L_\infty(\mathbb R_+^{2})}\iint_{\mathbb R_+^2}|Df(x,t)|^2t^{1-2\sigma}dxdt.$$
\item For $\sigma\in (0,1)$,
$$\iint_{\mathbb R_+^2}|Sf_t|^2t^{1-2\sigma}dxdt\leq c\iint_{\mathbb R_+^2}|f(x,t)|^2t^{1-2\sigma}dxdt.$$
\end{enumerate}
\end{lemma}
\begin{proof}
The estimates $(1)$ and $(2)$ follow from \cite[Thm 1.3]{R}. For the reader's convenience we outline the proof here.
Let $h_t:=\mathcal{E}_tDf_t$ and consider
\begin{eqnarray}\label{uu9uull1+lu}
\tilde Sh_t&=&\int_0^te^{-(t-s)\tilde \Lambda_0}\tilde E_0^+B_0h_sds+\int_t^\infty e^{-(s-t)\tilde\Lambda_0}\tilde E_0^-B_0h_sds.
\end{eqnarray}
Then \begin{eqnarray}\label{uu9uull1+lu+}
\tilde \Lambda_0^{\sigma}\tilde Sh_t&=&\int_0^t\tilde \Lambda_0^{1-\tilde \sigma}e^{-(t-s)\tilde \Lambda_0}\tilde E_0^+B_0h_sds\notag\\
&&+\int_t^\infty \tilde \Lambda_0^{1-\tilde \sigma}e^{-(s-t)\tilde\Lambda_0}\tilde E_0^-B_0h_sds
\end{eqnarray}
where $\tilde \sigma=1-\sigma$. Using Lemma \ref{lem1} we see that we want to estimate $\|\tilde \Lambda_0^{\sigma}\tilde Sh_t\|_2$. To do this we, following \cite{R}, write
\begin{eqnarray}\label{uu9uull1+lu+1}
\tilde \Lambda_0^{\sigma}\tilde Sh_t=I_1+I_2+I_3+I_4
\end{eqnarray}
where
\begin{eqnarray}\label{uu9uull1+lu+2}
I_1&=&\int_{t/2}^{2t}\tilde \Lambda_0^{1-\tilde \sigma}e^{-(t-s)\tilde \Lambda_0}\tilde E_0^{\mbox{sgn}(t-s)}B_0h_sds,\notag\\
I_2&=&\int_{0}^{t/2}\tilde \Lambda_0^{1-\tilde \sigma}e^{-(t-s)\tilde \Lambda_0}(I-e^{-2s\tilde\Lambda_0})\tilde E_0^{+}B_0h_sds,\notag\\
I_3&=&\int_{2t}^{\infty}\tilde \Lambda_0^{1-\tilde \sigma}e^{-(t-s)\tilde \Lambda_0}(I-e^{-2s\tilde \Lambda_0})\tilde E_0^{-}B_0h_sds,\notag\\
I_4&=&e^{-t\tilde \Lambda_0}\int_{\mathbb R\setminus[t/2,2t]}\tilde \Lambda_0^{1-\tilde \sigma}e^{-s\tilde \Lambda_0}\tilde E_0^{\mbox{sgn}(t-s)}B_0h_sds,
\end{eqnarray}
and where $\mbox{sgn}(t-s)$ is interpreted as $'+'$ when positive and as $'-'$ when negative. Consider first the case when
$\sigma\in [0,1/2)$, i.e., $\tilde \sigma\in (1/2,1]$. We then immediately see that
\begin{eqnarray}\label{uu9uull1+lu+3}
\|I_1\|_2&\lesssim&\int_{t/2}^{2t}\frac {\|h_s\|_2}{|t-s|^{1-\tilde \sigma}}ds\notag\\
&\lesssim&
\biggl (t^{2\sigma-1}\int_{t/2}^{2t}\frac {1}{|t-s|^{2-2\tilde \sigma}}ds\biggr )^{1/2}\|h\|_{L_2(\mathbb R_+^{2},t^{1-2\sigma})}\notag\\
&\lesssim&\|h\|_{L_2(\mathbb R_+^{2},t^{1-2\sigma})}.
\end{eqnarray}
To estimate $\|I_2\|_2$ we note that
\begin{eqnarray}\label{uu9uull1+lu+4}
&&\|\tilde \Lambda_0^{1-\tilde \sigma}e^{-(t-s)\tilde \Lambda_0}(I-e^{-2s\tilde\Lambda_0})\|=\notag\\
&&\|(s/(t-s)^{2-\tilde\sigma})((t-s)\tilde \Lambda_0)^{2-\tilde\sigma}e^{-(t-s)\tilde \Lambda_0}(I-e^{-2s\tilde\Lambda_0})/(s\tilde\Lambda_0\|
\lesssim s/t^{2-\tilde\sigma}.
\end{eqnarray}
Using this we see that
\begin{eqnarray}\label{uu9uull1+lu+5}
\|I_2\|_2&\lesssim& \int_{0}^{t/2}s/t^{2-\tilde\sigma}\|h_s\|_2ds\lesssim \|h\|_{L_2(\mathbb R_+^{2},t^{1-2\sigma})}.
\end{eqnarray}
Obviously a similar estimate holds for $\|I_3\|_2$. Finally, to estimate $\|I_4\|_2$ we consider $\phi\in L_2(\mathbb R)$, $\|\phi\|_2=1$, and note that
\begin{eqnarray}\label{uu9uull1+lu+6}
|(I_4,\phi)|\lesssim\int_0^\infty\|(s\tilde \Lambda_0)^{1-\tilde\sigma}e^{-s\tilde \Lambda_0^\ast}\phi\|_2\|s^{\tilde\sigma}h_s\|_2\frac {ds}s\lesssim
\|h\|_{L_2(\mathbb R_+^{2},t^{1-2\sigma})}.
\end{eqnarray}
These estimates complete the proof of $(1)$. The proof of $(2)$, in this case $\sigma\in [1/2,1)$, i.e., $\tilde \sigma\in (0,1/2]$, follow similar except
that in this case we we have to be slightly more careful when estimating
\begin{eqnarray}\label{uu9uull1+lu+7}
\frac 1 t\int_t^{2t}\|I_1(s)\|_2ds
\end{eqnarray}
where
\begin{eqnarray}\label{uu9uull1+lu+8}
I_1(s)&=&\int_{s/2}^{2s}\tilde \Lambda_0^{1-\tilde \sigma}e^{-(s-\tau)\tilde \Lambda_0}\tilde E_0^{\mbox{sgn}(s-\tau)}B_0h_\tau d\tau.
\end{eqnarray}
Indeed, in this case we have
\begin{eqnarray}\label{uu9uull1+lu+0}
\frac 1 t\int_t^{2t}\|I_1(s)\|_2ds&\lesssim& \frac 1 t\int_t^{2t}\biggl (\int_{s/2}^{2s}\frac {\|h_\tau\|_2}{|s-\tau|^{1-\tilde \sigma}}d\tau\biggr )ds\notag\\
&\lesssim& \frac 1 t\int_t^{2t}\biggl (\int_{s/2}^{2s}\frac {1}{|s-\tau|^{1-\tilde \sigma}}d\tau\biggr )
\biggl (\int_{s/2}^{2s}\frac {\|h_\tau\|_2^2}{|s-\tau|^{1-\tilde \sigma}}d\tau\biggr )ds\notag\\
&\lesssim& \frac 1 t\int_t^{2t}\tau^{\tilde\sigma}
\biggl (\int_{s/2}^{2s}\frac {\|h_\tau\|_2^2}{|s-\tau|^{1-\tilde \sigma}}d\tau\biggr )ds\notag\\
&\lesssim& \frac 1 t\int_{t/2}^{4t}\tau^{\tilde\sigma}
\biggl (\int_{\tau/2}^{2\tau}\frac {s^{\tilde\sigma}}{|s-\tau|^{1-\tilde \sigma}}ds\biggr )\|h_s\|_2^2ds
\lesssim
\|h\|_{L_2(\mathbb R_+^{2},t^{1-2\sigma})}.
\end{eqnarray}

To complete the proof of the lemma it only remains to prove statement $(3)$ of the lemma. However, this follows from \cite[Thm. 2.3]{R}, and the proof of Lemma \ref{lemmata2} is complete.
\end{proof}

  \setcounter{equation}{0} \setcounter{theorem}{0}
  \section{Proof of the main results}
  \noindent  In the following we let  $B_0$ be as in \eqref{a++a++ll3}
  and we recall the operators
introduced in \eqref{uu9uull} and acting in $L_2(\mathbb R)$. Recall that  $B_0D$ is an injective bisectorial operator and that
\begin{eqnarray}\label{uu9uull1}
L_2(\mathbb R)=\tilde E^+_0L_2(\mathbb R)\oplus \tilde E^-_0L_2(\mathbb R).
\end{eqnarray}

\begin{lemma}   \label{importantlemma}
Given $p$, $1<p<\infty$, and Lipschitz function $\phi$ with Lipschitz constant $M$, let $B(f)=B_p^{f,\phi}$ be as in \eqref{a++a++ll1}.
Then there exists $C(M,p)<\infty$ such that
$|B(f)|\le C(M,p)$ and
$$
  \text{Re} (B(f)v,v)\ge |v|^2/C(M,p), \qquad \text{for all } v\in \mathbb{C}^2,
$$
uniformly for all $f\ne 0$.
Moreover, for a fixed Lipschitz function $\phi$, we have
$$
  \limsup_{p\to 2}\left(\frac{\sup_{f\ne 0}|B(f)-B_0|}{ |p-2|}\right)<\infty.
$$
\end{lemma}
\begin{proof} The proof is straightforward and we omit it.\end{proof}

\begin{theorem} \label{thm0}
Let $0<\sigma<1$.
  The trace map, initially defined on $C_0^\infty(\overline{\R^2_+})$, extends uniquely  by continuity to a bounded
  linear operator
  $\dot H^1(\R^2_+,t^{1-2\sigma})\to \dot H^\sigma(\R)$
  with bounds
\begin{equation}   \label{eq:tracemap}
 \iint_{\R^2} \frac{|f(x,0)-f(x',0)|^2}{|x-x'|^{1+2\sigma}}dxdx' \le c \iint_{\R^2_+} |\nabla f(x,t)|^2 t^{1-2\sigma}dxdt.
\end{equation}
In particular, for any $f\in \dot H^1(\R^2_+,t^{1-2\sigma})$, we have limits
$\lim_{t\to 0^+} f_t=f_0$ and $\lim_{t\to \infty}f_t=0$ in $\dot H^\sigma(\R)$.
\end{theorem}

\begin{proof} Given a smooth and compactly supported function $f$, and writing
\begin{eqnarray*}
|f(x,0)-f(x',0)|&\le& |f(x,0)-f(x,t)|+|f(x,t)-f(x',t)|\notag\\
&&+|f(x',t)-f(x',0)|,
\end{eqnarray*}
 with $t:= |x-x'|$, it is straightforward
to establish the stated trace bound. Indeed, the bound follows by applying a Hardy estimate, with weight $t^{1-2\sigma}$ in the variable $t$, to the first and last terms, and Plancherel's identity in the variable $h=x-x'$ to the middle term. The density of, and the extension from, $C_0^\infty(\overline{\R^2_+})$ is proved as follows. By Chua \cite{Ch}, it suffices to prove density of $C^\infty_0(\R^2)$ in the corresponding weighted
Sobolev $\dot H^1(\R^2,|t|^{1-2\sigma})$ space on $\R^2$. Standard mollification together with a cutoff argument, using
the weighted Poincar\'e inequality of Fabes, Kenig and Serapioni \cite{FKS} applies.
Note that $|t|^{1-2\sigma}$ is an $A_2$ weight on $\R^2$. \end{proof}

\subsection{Proof of Theorem \ref{thm1-}}
Let $p$, $1<p<\infty$, be given and let $\sigma\in (0,1)$. Let $\phi:\mathbb R\to\mathbb R$ be a Lipschitz function with constant at most $M$.  Assume that $u$ is $p$-harmonic in  $\Omega=\{(x,y):\ x\in\mathbb R,\ y>\phi(x)\}$ and that
\begin{eqnarray}
\iint_\Omega|\nabla ^2u|^2(y-\phi(x))^{1-2\sigma}dxdy<\infty.
\end{eqnarray}
Let $(x,y)\to (x,t)$, $t=y-\phi(x)$, $f(x,t)=(\partial_xu(x,y),-\partial_yu(x,y))^\ast$. Then,
\begin{eqnarray}\label{a++a++bla+sek1}
\partial_{t}f+B_0Df=B_0\mathcal{E}Df,\ \mathcal{E}:=I-B_0^{-1}B_p^{f,\phi}.
\end{eqnarray}
We note that $\mathcal{E}$ is a multiplier, defined almost everywhere by Corollary~\ref{lemmata1+}, which is bounded by Lemma~\ref{importantlemma}, with bound independent of $f$. Using \eqref{uu9uull1+hha} we see that
\begin{eqnarray}\label{a++a++bla+ksek2}
f_t=e^{-tB_0D}\tilde E_0^{+}g^++\tilde S\mathcal{E}_tDf_t.
\end{eqnarray}
Formally, we expect
$$
  \|g^+\|^2_{\dot H^{\sigma}(\mathbb R)}\lesssim \iint_\Omega|\nabla ^2u|^2(y-\phi(x))^{1-2\sigma}dxdy
$$
from \eqref{si3}. This is indeed the case, and this can be proved rigorously as in \cite{R}. We refer to \cite{R}  for detailed proofs. See Lemma \ref{lemmata1} and Lemma \ref{lemmata2} for the main estimates.

\subsection{Proof of Theorem \ref{thm1}} Let $\sigma\in (0,1)$ and a Lipschitz function $\phi$, with Lipschitz constant $M$, be given. Consider $h\in  \dot H^\sigma(\mathbb R)$, which we without loss of generality assume is non-constant.
Consider the linear boundary value problem
\begin{equation}  \label{bvplin}
\begin{cases}
  \partial_t f+ B D f=0, & \qquad \text{in } \R^2_+, \\
  f_1 = h, & \qquad \text{on } \R.
\end{cases}
\end{equation}
Assume that
$|p-2|<\delta$, where $\delta>0$ can and is chosen so that the boundary value problem in \eqref{bvplin}
is well posed whenever $\|B-B_0\|_{L_\infty(\R^2_+)}<\delta$. To prove Theorem \ref{thm1} we intend to apply the following version of Schauder's fixed point theorem by Singbal. For a
proof, see \cite{B}.

\begin{theorem}\label{fix}
  Let $V$ be a locally convex topological vector space, let $X$ be a non-empty
  closed convex subset of $V$, and let $T$ be a continuous mapping of $X$ into a compact
  subset of $X$. Then $T$ has a fixed point, i.e.,  $T(x)=x$ for some $x\in X$.
\end{theorem}

In the proof of Theorem \ref{thm1} our space $V$ will be $L_\infty(\R^2_+, \R^{2\times 2})$, equipped with semi-norms
$$
  p_f(B):= \left( \iint_{\R^2_+} |Bf|^2 t^{1-2\sigma} dxdt \right)^{1/2}.
$$
I.e.,  $V$ is the subspace of multipliers in the space of bounded linear operators
on $L_2(\R^2_+,t^{1+2\sigma})$, equipped with the strong operator topology. For the set $X$ we choose to be the closed ball $X:= \{ B\in L_\infty(\R^2_+,\R^{2\times 2}):\|B-B_0\|_\infty<r\}$ around $B_0$ in $V$, for some small radius $r<\delta$,
and the map $T$ is set to be
$$
  T: B\mapsto B(f),
$$
where $f$ solves the linear boundary value problem \eqref{bvplin} with coefficients $B$,
and $B(f)$ is calculated
from $f$ pointwise as in \eqref{a++a++ll1}-\eqref{a++a++llkkuu}.

Using this set up we see that to prove Theorem \ref{thm1} it remains to verify the hypothesis of Theorem \ref{fix}. In particular, we have to verify that
$T$ is a continuous mapping of $X$ into a compact subset of $X$.

To establish the continuity of $T$, we use the Cauchy integral representation
from Theorem \ref{thm1-} and the operators and estimates from Section~\ref{fcal} as follows. Consider a strongly convergent sequence $B_n\to B$ of
coefficients in $X$.
Consider the auxiliary linear operator $P^B:\dot H^\sigma(\R)\to \dot H^1(\R^2_+,t^{1-2\sigma})$ given by
\begin{eqnarray}\label{rel1}
  g\mapsto P^Bg:= S_0 g+\tilde S\mathcal{E} (I-S\mathcal{E})^{-1}DS_0 g
\end{eqnarray}
where $\mathcal{E}=I-(B_0)^{-1}B$. For any auxiliary boundary function $g$, this gives a solution $f:= P^Bg$ to the
linear PDE $\partial_t f_t+ B_tDf_t=0$.
It is straightforward to verify that the operators $P^{B_n}$ converge strongly to $P^B$ as $B_n\to B$.
To solve the boundary value problem, we solve the boundary equation $(P^B_0 g)_1=h$, where
$P^B_0g:= \lim_{t\to 0^+}(P^Bg)_t$.
The operators $P^B_0$ are seen to be projections onto the subspace of $\dot H^\sigma(\R)$ consisting
of traces of solutions to $\partial_t f_t + B_t D f_t=0$ in $\R^2_+$.
One then verifies, see \cite{R} and \cite[Lem. 4.3]{AAMc}, that the functions $g_n$ corresponding
to $B_n$ will converge in $\dot H^\sigma(\R)$, as a consequence of the strong convergence of $P^{B_n}_0$ as $B_n\to B$, to the function $g$ corresponding to $B$.
In total, we obtain solutions $f_n= P^{B_n}g_n$ which converge to $f= P^B g$ in $ \dot H^1(\R^2_+;t^{1-2\sigma})$.
Poincar\'e's inequality and the interior estimates of Lemma \ref{rem1} show that
$f_n\to f$ locally uniformly. From the dominated convergence theorem, using Lemma~\ref{importantlemma}, it follows that $B(f_n)=T(B_n)$ converges to
$B(f)= T(B)$ in $V$.

To establish the compactness of the image of $T$ we first note that it suffices, since $V$ is metrizable, to consider a sequence of solutions $f_k$
to linear problems $\partial_t f_k+ B_k D f_k=0, (f_k)_1=h$,
with coefficients $B_k\in X$. However, using the bounds
$$
  \|f_k\|_{\dot H^1(\R^2_+;t^{1+2\sigma})}\le c \|h\|_{\dot H^\sigma(\R)},
$$
we obtain, using
Poincar\'e's inequality, the interior estimates of Lemma \ref{rem1}, and
Arzel\`a--Ascoli's theorem, the existence of a subsequence $f_{k_j}$ which
converges locally uniformly to some limit function $f$, which is quasiregular by Lemma \ref{th2}.
We note that $f$ is non-constant by Theorem \ref{thm0}, since the trace of $f$ is $h$.
The union $Z\subset \R^2_+$
of the zeros of $f$ and all $f^{k_j}$, $j=1,2,\ldots$, has measure zero by (1) of Corollary \ref{lemmata1+}.
Thus
$B(f^{k_j})$ converges pointwise to $B(f)$ on $\R^2_+\setminus Z$.
Using Lemma~\ref{importantlemma},
the dominated convergence theorem applies and shows that $B(f_{k_j})$ converges to $B(f)$ in
the topology of $V$.
This completes the proof of Theorem \ref{thm1}.

   \setcounter{equation}{0} \setcounter{theorem}{0}
   \section{Concluding remarks}
   \noindent
   We here briefly discuss generalizations of Theorem \ref{thm1-} and Theorem \ref{thm1} to the more general
   quasi-linear PDEs in the plane considered in section 2. Indeed, consider \eqref{generalHuu} assuming \eqref{asp} and recall
   Lemma \ref{reduk}. The lemma states that $u$ is a weak solution to \eqref{generalHuu} in $\Omega$ if and only if
$$f(x,t)=(f_1(x,t), f_2(x,t)):=\big(\partial_xu(x,t+\phi(x)),-\partial_yu(x,t+\phi(x))\big)^\ast$$ satisfies
\begin{eqnarray}\label{a++a++bla+luha}
\partial_{t}\begin{bmatrix}
f_1\\
f_2
\end{bmatrix}+B_p^{a,f,\phi}D\begin{bmatrix}
 f_1\\
f_2
\end{bmatrix}=0,
\end{eqnarray}
in $\mathbb R_+^2:=\{(x,t)\in\mathbb R^2:\ t>0\}$ where we now stress the dependence of $B_p^{f,\phi}$ on the symbol $a$ by writing $B_p^{a,f,\phi}$.
Furthermore, $B_p^{a,f,\phi}$ is given in \eqref{a++a++ll}, \eqref{a++a++ll+} and by Lemma \ref{reduk+} we have $B_p^{a,f,\phi}\in L_\infty(\mathbb R^{2}_+,\mathbb C^2)$ and $B$
is accretive in the sense of \eqref{uu3}. Furthermore, the $L_\infty$-bound on $B_p^{a,f,\phi}$, and the parameter of accretivity $\kappa$, depend only on $p$, $M$,  $\nu$, and $L$.  In the following we let  $B_0$ be as in \eqref{a++a++ll3}
  and we recall the operators
introduced in \eqref{uu9uull} and acting in $L_2(\mathbb R)$. Then, as discussed, $B_0D$ is an injective bisectorial operator and
\begin{eqnarray}\label{uu9uull1}
L_2(\mathbb R)=\tilde E^+_0L_2(\mathbb R)\oplus \tilde E^-_0L_2(\mathbb R).
\end{eqnarray}
Let $\epsilon\in(0,1)$, consider $p$, $1<p<\infty$, fixed and let
$\phi$ be a fixed Lipschitz function with Lipschitz constant $M$. Then, in the following we say that $B_p^{a,f,\phi}$ is within $\epsilon$ of $B_0$ if
$$
\sup_{f\ne 0} {|B_p^{a,f,\phi}-B_0|}<\epsilon.
$$
Note that when the underlying operator is the $p$-Laplace operator, i.e., $a(\eta)=|\eta|^{p-2}\eta$, and $B(f)=B_p^{a, f,\phi}$ as in \eqref{a++a++ll1}, then Lemma \ref{importantlemma} states that $$
  \limsup_{p\to 2}\frac{\sup_{f\ne 0}|B_p^{a,f,\phi}-B_0|}{ |p-2|}<\infty
$$
and hence in this case $B_p^{a,f,\phi}$ is within $\epsilon$ of $B_0$, for any $\epsilon\in(0,1)$, as long as $|p-2|$ is small enough. The following two theorems, generalizing Theorem \ref{thm1-} and Theorem \ref{thm1} to the more general
   quasi-linear PDEs in the plane, can be proved by repeating the arguments in section 4.
   \begin{theorem}    \label{thm1-a}
Let $1<p<\infty$, $0<\sigma <1$, and $0\le M<\infty$ be given.
Assume that $\phi:\mathbb R\to\mathbb R$ is a Lipschitz function with $\|\phi'\|_\infty\le M$
and assume that $u$ is a weak solution to \eqref{generalHuu}, assuming \eqref{asp}, in $\Omega=\{(x,y):\ x\in\mathbb R,\ y>\phi(x)\}$ satisfying
\begin{eqnarray}
\iint_\Omega|\nabla ^2u|^2(y-\phi(x))^{1-2\sigma}dxdy<\infty.
\end{eqnarray}
Let $f(x,t)=(\partial_xu(x,y),-\partial_yu(x,y))^\ast$, $y=t+\phi(x)$.
Then there exists $g\in \dot H^\sigma(\R)$ such that the Cauchy integral representation
\begin{eqnarray}\label{a++a++bla+kla}
f= S_0 g +\tilde S ((B_0-B_p^{a,f,\phi}) Df),
\end{eqnarray}
holds in $\R^2_+$. In particular, the
 trace of $f$ is $$f_0= \lim_{t\to 0^+}(S_0g) -\int_0^\infty \big( S_0(B_0-B_p^{a,f,\phi}) Df_s\big)_{-s} ds  \in \dot H^\sigma(\R).$$
\end{theorem}
\begin{theorem}\label{thm1a}
Let $p, \sigma, M, \phi, $ be as in Theorem~\ref{thm1-a}. Then there  exists $\epsilon_0=\epsilon_0(p,\sigma,M)$, $\epsilon_0\in(0,1)$, such that the following is true. Let $\epsilon\in(0,\epsilon_0)$ and assume that $B_p^{a,f,\phi}$ is within $\epsilon$ of $B_0$. Then, given any boundary data $h\in  \dot H^\sigma(\mathbb R)$, there exists
a weak solution $u$ to \eqref{generalHuu} in $\Omega=\{(x,y):\ x\in\mathbb R,\ y>\phi(x)\}$ satisfying
\begin{eqnarray}
\iint_\Omega|\nabla ^2u|^2(y-\phi(x))^{1-2\sigma}dxdy<\infty,
\end{eqnarray}
and the boundary condition
$$
  \partial_x u(x,\phi(x))= h(x), \qquad x\in\mathbb R,
$$
where the trace of $\nabla u$ is taken in the  sense of Theorem \ref{thm0}. The same solvability result also holds true for the
boundary condition $\partial_y u(x,\phi(x))= h(x)$.
\end{theorem}

\end{document}